\documentclass{proc-l}
\usepackage{amsmath,amssymb}

\newtheorem{theorem}{Theorem}[section]

\newtheorem{corollary}[theorem]{Corollary}

\numberwithin{equation}{section}

\newcommand{\R}{\mathbb{R}}
\newcommand{\N}{\mathbb{N}}

\newcommand{\cF}{\mathcal{F}}
\newcommand{\cM}{\mathcal{M}}

\newcommand{\supp}{\textup{supp}}

\newcommand{\diam}{\textup{diam}}

\newcommand{\eps}{\varepsilon}

\begin{document}

\title[Distribution of point charges with small discrete energy]
{Distribution of point charges with small discrete energy}%
\author{Igor E. Pritsker}%

\thanks{Research was partially supported by the National Security
Agency, and by the Alexander von Humboldt Foundation.}

\address{Department of Mathematics, Oklahoma State University, Stillwater, OK 74078, U.S.A.}%
\email{igor@math.okstate.edu}




\begin{abstract}

We study the asymptotic equidistribution of points near arbitrary compact sets of positive capacity in $\R^d,\ d\ge 2$. Our main tools are the energy estimates for Riesz potentials. We also consider the quantitative aspects of this equidistribution in the classical Newtonian case. In particular, we quantify the weak convergence of discrete measures to the equilibrium measure, and give the estimates of convergence rates for discrete potentials to the equilibrium potential.

\end{abstract}

\maketitle


\section{Asymptotic equidistribution of discrete sets}

Let $E$ be a compact set in $\R^d,\ d\ge 2.$  Denote the Euclidean distance between $x\in\R^d$ and $y\in\R^d$ by $|x-y|$. We consider potential theory associated with Riesz kernels
\[
k_{\alpha}(x):=|x|^{\alpha-d},\quad x\in\R^d,\quad 0<\alpha<d.
\]
For a Borel measure $\mu$ with compact support, define its energy by
\[
I_{\alpha}[\mu]:= \iint k_{\alpha}(x-y) \, d \mu(x)d \mu(y).
\]
A central theme in potential theory is the study of the minimum energy problem
\[
W_{\alpha}(E):=\inf_{\mu\in\cM(E)} I_{\alpha}[\mu],
\]
where $\cM(E)$ is the space of all positive unit Borel measures supported on $E$. If Robin's constant $W_{\alpha}(E)$ is finite, then the above infimum is attained by the equilibrium measure $\mu_E\in\cM(E)$ \cite[p. 131--133]{La}, which is a unique probability measure expressing the steady state distribution of charge on the conductor $E$. The capacity of $E$ is defined by
\[
C_{\alpha}(E):=\frac{1}{W_{\alpha}(E)},
\]
where we set $C_{\alpha}(E)=0$ when $W_{\alpha}(E)$ is infinite. For a more detailed exposition of Riesz potential theory, we refer the reader to the book of Landkof \cite{La}.

The main goal of this paper is a study of discrete approximations to the equilibrium measure. Consider the counting measure $\tau(X_n)$ for a discrete set $X_n=\{x_{k,n}\}_{k=1}^n\subset\R^d$, given by
\[
\tau(X_n) := \frac{1}{n} \sum_{k=1}^n \delta_{x_{k,n}},
\]
where $\delta_{x_{k,n}}$ is the unit point mass at $x_{k,n}\in X_n$.
We define the {\em discrete} energy of $\tau(X_n)$ (or of the set $X_n$) by setting
\[
\hat{I}_{\alpha}[\tau(X_n)] := \frac{2}{n(n-1)} \sum_{1\le j<k\le n} k_{\alpha}(x_{j,n}-x_{k,n}).
\]
A set of points $\cF_n\subset E$ that minimizes the above energy among all $n$-tuples from $E$ is called the $n$-th Fekete points of $E$. The Fekete-Szeg\H{o} results on the transfinite diameter suggest that
\[
\lim_{n\to\infty} \inf_{X_n\subset E} \hat{I}_{\alpha}[\tau(X_n)] = \lim_{n\to\infty} \hat{I}_{\alpha}[\tau(\cF_n)] =  W_{\alpha}(E) = I_{\alpha}[\mu_E],
\]
which simply indicates that the discrete approximations of the minimum energy converge to Robin's constant. Furthermore, the counting measures $\tau(\cF_n)$ converge weakly to $\mu_E$  (written $\tau(\cF_n) \stackrel{*}{\rightarrow} \mu_E$) as $n\to\infty$,
provided that $W_{\alpha}(E)$ is finite, cf. \cite[pp. 160-162]{La}. Such equidistribution property holds for many sequences of discrete sets whose energies converge to Robin's constant, which gives rise to numerous possibilities of how one may discretize the equilibrium measure. These ideas originated in the work of Fekete \cite{Fe} and Szeg\H{o} \cite{Sz} for logarithmic potentials in the plane. One can find an extensive discussion of related questions in Andrievskii and Blatt \cite{AB}, including  history and references. The study of discrete Riesz potentials gained momentum more recently, and the area remains quite popular, see the surveys by Korevaar \cite{Ko}, Saff and Kuijlaars \cite{SK}, and Hardin and Saff \cite{HS}. We proved certain general qualitative and quantitative results for the discrete approximations of equilibrium measures in the plane \cite{Pr}. In the present paper, we extend the energy methods used in \cite{Pr} to Riesz potentials.

Define the Riesz potential of a Borel measure $\mu$ with compact support in $\R^d$ by
\[
U_{\alpha}^{\mu}(x) := \int k_{\alpha}(x-y)\,d\mu(y),\quad x\in\R^d.
\]
Note that $U_{\alpha}^{\mu}$ is a superharmonic function in $\R^d$ for $2\le\alpha<d$, and $U_{\alpha}^{\mu}$ is subharmonic in $\R^d\setminus\supp(\mu)$ for $0<\alpha\le 2$. Thus in the classical case $\alpha=2$, the Newtonian potential $U_2^{\mu}$ is harmonic in $\R^d\setminus\supp(\mu)$. If $W_{\alpha}(E)<\infty$ then the equilibrium measure $\mu_E$ exists, and we define the function $g_E(x)$ by
\begin{align*}
g_E(x) = W_{\alpha}(E) - U_{\alpha}^{\mu_E}(x),\quad x\in\R^d.
\end{align*}
When $0<\alpha\le 2$, $g_E$ is a nonnegative upper semi-continuous  function in $\R^d\cup\{\infty\}$, which is superharmonic on $E^c:=(\R^d\cup\{\infty\})\setminus E$, see \cite[p. 137]{La}. In the Newtonian case $\alpha=2$, $g_E$ coincides with the classical Green function for the unbounded component of $E^c.$ We use the quantity
\[
m_E(X_n) := \frac{1}{n} \sum_{x_{k,n}\in E^c} g_E(x_{k,n})
\]
to measure how close $X_n$ is to $E$. If $X_n\subset E$ then we set $m_E(X_n)=0$ by definition.

\begin{theorem} \label{thm1.1}
Let $0<\alpha\le 2$, and let $E\subset\R^d$ be a compact set with finite Robin's constant $W_{\alpha}(E)$. If the sets $X_n=\{x_{k,n}\}_{k=1}^n\subset\R^d,\ n\ge 2,$ satisfy
\begin{align} \label{1.1}
\lim_{n\to\infty} \hat{I}_{\alpha}[\tau(X_n)] = W_{\alpha}(E)
\end{align}
and
\begin{align} \label{1.2}
\lim_{n\to\infty} m_E(X_n) = 0,
\end{align}
then
\begin{align} \label{1.3}
\tau(X_n) \stackrel{*}{\rightarrow} \mu_E \mbox{ as }n\to\infty.
\end{align}
Conversely, \eqref{1.2} holds for any sequence of the sets $X_n=\{x_{k,n}\}_{k=1}^n\subset\R^d,\ n\in\N,$ satisfying \eqref{1.3}.
\end{theorem}

When $X_n\subset E,$ we clearly have that $m_E(X_n)=0$ for all $n\ge 2,$ and \eqref{1.1} implies the well known fact that $\tau(X_n) \stackrel{*}{\rightarrow} \mu_E \mbox{ as }n\to\infty,$ see \cite[pp. 161-162]{La}. A new feature of the above result is that $X_n$ is not required to be a subset of $E$, allowing discretization schemes with point charges located outside $E$.

We remark that Theorem \ref{thm1.1} is valid in a more general setting, where the Riesz kernels $k_\alpha$ are replaced by kernels of the form $K(r)=H(-\log{r})$ when $d=2$, and $K(r)=H(r^{2-d})$ when $d\ge 3$.
One should assume that $H:\R\to[0,\infty)$ is a continuous increasing strictly convex function, and use a version of potential theory developed in Carleson \cite{Ca}, and Aikawa and Ess\'en \cite{AE}. The proof of this more general result closely follows our proof for Riesz kernels with $0<\alpha<2$.

\section{Rate of convergence and discrepancy in equidistribution}

Theorem \ref{thm1.1} describes conditions guaranteeing that the counting measures $\tau(X_n)$ converge to the equilibrium measure $\mu_E$ as $n\to\infty.$ This section is devoted to the estimates of rates in this convergence. The estimates of how close $\tau(X_n)$ is to the equilibrium measure $\mu_E$ are often called discrepancy estimates. We shall only consider the Newtonian case $\alpha=2$ in $\R^d,\ d\ge 3$. Logarithmic potentials in the plane were studied in \cite{Pr}, and we generalize the ideas of \cite{Pr} here. We shall suppress the subscript $\alpha=2$ in the notation, and write $U^{\mu}(x):=U_2^{\mu}(x)$, $I[\mu]:=I_2[\mu]$, $W(E):=W_2(E),$ etc. Consider a class of continuous functions $\phi:\R^d\to\R$ with compact support in $\R^d,\ d\ge 3.$ Since $\tau(X_n) \stackrel{*}{\rightarrow} \mu_E \mbox{ as }n\to\infty$ means
\[
\lim_{n\to\infty} \frac{1}{n} \sum_{k=1}^n \phi(x_{k,n}) = \lim_{n\to\infty} \int\phi\,d\tau(X_n) = \int\phi\,d\mu_E,
\]
it is most natural to seek the quantitative estimates of convergence $\tau(X_n) \stackrel{*}{\rightarrow} \mu_E$ in terms of convergence rates of the above $\phi$-means to $\int\phi\,d\mu_E.$ One may view this approach as a study of approximate quadrature rules for $\int\phi\,d\mu_E.$ Let
\[
\omega(\phi;r):=\sup_{|x-y|\le r} |\phi(x)-\phi(y)|
\]
be the modulus of continuity of $\phi$ in $\R^d$. We also require that functions $\phi$ have finite Dirichlet integral
\[
D[\phi]:= \int_{\R^d} |\nabla\phi|^2\,dV = \int_{\R^d} \sum_{i=1}^d \left(\frac{\partial\phi}{\partial x_i}\right)^2\,dV(x),\quad x=(x_1,\ldots,x_d)\in\R^d,
\]
where it is assumed that the partial derivatives of $\phi$ exist a.e. on $\R^d$ in the sense of $d$-dimensional Lebesgue measure $dV(x)$. We denote the surface area of the unit $(d-1)$-dimensional sphere in $\R^d$ by $\omega_d:=2\pi^{d/2}/\Gamma(d/2).$ Define the distance from a point $x\in\R^d$ to a compact set $E$ by
\[
d_E(x):=\min_{t\in E} |x-t|.
\]

\begin{theorem} \label{thm2.1}
Let $E\subset\R^d,\ d\ge 3,$ be a compact set with $W(E)<\infty$, and let $\phi:\R^d\to\R$ be a continuous function with compact support such that $D[\phi]<\infty.$ If $X_n=\{x_{k,n}\}_{k=1}^n\subset\R^d,\ n\ge 2,$ then we have for any $r>0$ that
\begin{align} \label{2.1}
\left|\frac{1}{n} \sum_{k=1}^n \phi(x_{k,n}) - \int\phi\,d\mu_E\right| \le \omega(\phi;r) + \sqrt{\frac{D[\phi]}{(d-2)\omega_d}}\, \sqrt{I},
\end{align}
where
\begin{align} \label{2.2}
I = 2m_E(X_n) + \frac{n-1}{n} \hat{I}[\tau(X_n)] - W(E) + \frac{r^{2-d}}{n} + 2\max_{d_E(x)\le 2r} g_E(x).
\end{align}
\end{theorem}

Energy ideas have been used in discrepancy estimates by Kleiner \cite{Kl}, Sj\"{o}gren \cite{Sj1}-\cite{Sj2}, Huesing \cite{Hus} and G\"otz \cite{Go1}-\cite{Go3}, see \cite[Ch. 5]{AB}. A typical application of our result is given by a sequence of sets $X_n$ satisfying \eqref{1.1} and \eqref{1.2}. If we choose $r=r_n\to 0$ as $n\to\infty$, then the right hand side of \eqref{2.1} tends to 0 under the assumption that the Green function $g_E(x)$ is continuous at the boundary points of $\Omega_E$ (i.e. $E$ is regular). In order to obtain polynomial rates of convergence, one should set $r_n=c/n^a,$ with $a,c>0,$ and consider sets with uniformly H\"older continuous Green functions. The condition of the uniform H\"older continuity for $g_E(x)$ means that
\begin{align} \label{2.3}
g_E(x) \le A(E) (d_E(x))^s, \quad x\in\Omega_E,
\end{align}
where $A(E)>0$ and $0<s\le 1$ are independent of $x\in\Omega_E$. Note that the set $E$ need not be smooth for \eqref{2.3} to hold. In fact, \eqref{2.3} is satisfied for quite general classes of sets. The problem of H\"older continuity of Green functions was studied by Carleson and Totik \cite{CT}, Maz'ja \cite{Ma1}-\cite{Ma2}, To\'okos \cite{Too}, and Totik \cite{To}.

It is clear that various choices of $\phi$ lead to diverse applications of Theorem \ref{thm2.1}. We consider an application of Theorem \ref{thm2.1} to the potentials of ``near-Fekete" points, i.e., to the potentials of sets $X_n\subset E$ whose discrete energies are close to $W(E)$. Selecting $\phi$ as a modification of the kernel $k_2(x-y)$, we show that the potentials of discrete measures $\tau(X_n)$ are close to the equilibrium potential.

\begin{theorem} \label{thm2.2}
Let $E\subset\R^d,\ d\ge 3,$ be a compact set with $W(E)<\infty,$ such that the H\"older condition \eqref{2.3} holds with an exponent $s\in(0,1]$. Suppose that  $X_n=\{x_{k,n}\}_{k=1}^n\subset E,\ n\ge 2,$ satisfy
\begin{align} \label{2.4}
\hat{I}[\tau(X_n)] - W(E) \le C_1\,n^{-p}, \quad n\ge 2,
\end{align}
where $p:=s/(d+s-2)$ and $C_1>0$ is independent of $X_n$. Then we have for any $y\in E^c$ and any $n\ge 2$ that
\begin{align} \label{2.5}
\left|U^{\mu_E}(y) - U^{\tau(X_n)}(y)\right| \le C_2\left((d_E(y))^{1-d} n^{-p/s} + (d_E(y))^{1-d/2} n^{-p/2} \right),
\end{align}
where $C_2>0$ is independent of $y$ and $X_n$. Furthermore, there exists $q=q(d,s)>0$ such that
\begin{align} \label{2.6}
\sup_{y\in \R^d} \left(U^{\mu_E}(y) - U^{\tau(X_n)}(y)\right) \le C_3 n^{-q}, \quad n\ge 2,
\end{align}
where $C_3>0$ is independent of $X_n$.
\end{theorem}

The Fekete points $\cF_n=\{\zeta_{k,n}\}_{k=1}^n$ often represent the most natural way to discretize the equilibrium measure. However, they are difficult to find explicitly and even numerically, as all points of $\cF_n$ change with $n$. A more convenient choice of discretization frequently used in practice is given by Leja points, which are defined as a sequence. If $E\subset\R^d$ is a compact set of positive capacity, then the Leja (or Leja-G\'orski) points $\{\xi_k\}_{k=0}^{\infty}$ are defined recursively in the following way.  We choose $\xi_0\in E$ as an arbitrary point. When $\{\xi_k\}_{k=0}^n$ are selected, we choose the next point $\xi_{n+1}\in E$ as a point satisfying
\[
\sum_{k=0}^{n} |\xi_{n+1}-\xi_k|^{2-d} = \min_{x\in E} \sum_{k=0}^{n} |x-\xi_k|^{2-d}.
\]
It is known that Leja points are equidistributed in $E$. Theorem \ref{thm2.2} provides new quantitative information about discrete potentials of Fekete and Leja points for non-smooth sets.

\begin{corollary} \label{cor2.3}
If $E\subset\R^d,\ d\ge 3,$ is a compact set satisfying \eqref{2.3}, then \eqref{2.5} and \eqref{2.6} hold true for the Fekete and Leja points of $E$.
\end{corollary}

Surveys of results on Fekete points may be found in Korevaar \cite{Ko}, Andrievskii and Blatt \cite{AB} and Korevaar and Monterie \cite{KM}. We note that the estimates of Theorem \ref{thm2.2} can be improved for the Fekete points of a set $E$ satisfying more restrictive smoothness conditions. Results on Leja points may be found in G\"otz \cite{Go3}.

\section{Proofs}

We briefly review some well known facts from Riesz potential theory for $0<\alpha\le 2,$ see \cite{La}. If $U_{\alpha}^{\mu_E}(x)$ is the equilibrium (conductor) potential for $E$, then \cite[p. 137]{La}
\begin{align} \label{5.1}
0 < U_{\alpha}^{\mu_E}(x) \le W_{\alpha}(E),\ x\in\R^d, \quad \mbox{and} \quad U_{\alpha}^{\mu_E}(x) = W_{\alpha}(E) \mbox{ q.e. on }E.
\end{align}
The second statement means that equality holds quasi everywhere on $E$, i.e., except for a subset of zero capacity in $E$. Thus the function $g_E(x)=W_{\alpha}(E) - U_{\alpha}^{\mu_E}(x)$ satisfies
\begin{align} \label{5.2}
0 \le g_E(x) \le W_{\alpha}(E),\quad x\in\R^d\cup\{\infty\}.
\end{align}
If $0<\alpha<2$ then subharmonicity of $g_E$ and strict convexity of the kernel $k_{\alpha}$ imply that $g_E(x)>0$ for $x\in E^c.$ The Newtonian case $\alpha=2$ is special. Let $\Omega_E$ be the unbounded connected component of $E^c=(\R^d\cup\infty)\setminus E$. For $\alpha=2,$ the equilibrium measure $\mu_E$ is supported on $\partial\Omega_E$. As a result, $g_E$ is subharmonic in $\R^d$ and harmonic in $\R^d\setminus\partial\Omega_E$. Furthermore, $g_E$ is strictly positive on $\Omega_E$, and is identically zero on $\R^d\setminus\overline\Omega_E$ (hence also zero on bounded components of $E^c$). Note that $g_E(x)$ coincides with the Green function of $\Omega_E$ for $x\in\Omega_E$.

\begin{proof}[Proof of Theorem \ref{thm1.1}]
Set $\tau_n:=\tau(X_n)$ for brevity. We first prove that \eqref{1.1} and \eqref{1.2} imply \eqref{1.3}. Observe that each closed set $F\subset\Omega_E$ contains $o(n)$ points of $X_n$ as $n\to\infty,$ i.e.
\begin{align} \label{5.3}
\lim_{n\to\infty} \tau_n(F) = 0.
\end{align}
This fact follows from \eqref{1.2} because $\min_{x\in F} g_E(x) > 0$ and
\[
0 \le \tau_n(F) \min_{x\in F} g_E(x) \le  \frac{1}{n} \sum_{x_{k,n}\in F} g_E(x_{k,n}) \le m_E(X_n) \to 0 \quad\mbox{as } n\to\infty.
\]
The same argument implies for $0<\alpha<2$ that \eqref{5.3} holds for any closed set $F\subset E^c.$ Thus if $R>0$ is sufficiently large, so that $E\subset B_R:=\{x:|x|<R\},$ we have $o(n)$ points of $X_n$ in $\R^d\setminus B_R.$ Consider
\[
\hat\tau_n := \frac{1}{n} \sum_{|x_{k,n}|<R} \delta_{x_{k,n}}= \tau_n\vert_{B_R}.
\]
Since $\supp(\hat\tau_n)\subset B_R,\ n\in\N,$ we use Helly's theorem to select a weakly convergent subsequence from the sequence $\hat\tau_n$. Preserving the same notation for this subsequence, we assume that $\hat\tau_n \stackrel{*}{\rightarrow} \tau$ as $n\to\infty$. It is also clear from \eqref{5.3} that $\tau_n \stackrel{*}{\rightarrow} \tau$  as $n\to\infty$. Furthermore, $\tau$ is a probability measure supported on the compact set $\hat E := \R^d\setminus\Omega_E$ for $\alpha=2$, and on $E$ for  $0<\alpha<2.$ Suppose that $R>0$ is large, and order $x_{k,n}$ as follows
\[
|x_{1,n}| \le |x_{2,n}| \le \ldots \le |x_{m_n,n}| < R \le |x_{m_n+1,n}| \le \ldots \le |x_{n,n}|.
\]
Then
\begin{align} \label{5.4}
\hat{I}_{\alpha}[\tau_n]  &= \hat{I}_{\alpha}[\hat \tau_n] + \frac{2}{n(n-1)} \sum_{1\le j<k \atop m_n<k\le n} k_{\alpha}(x_{j,n}-x_{k,n}) \ge \hat{I}_{\alpha}[\hat \tau_n],
\end{align}
where we used that $k_{\alpha}(x)>0$ for all $x\in\R^d.$ Thus we obtain from \eqref{5.4} and \eqref{1.1} that
\begin{align} \label{5.5}
\limsup_{n\to\infty} \hat{I}_{\alpha}[\hat \tau_n] \le \limsup_{n\to\infty} \hat{I}_{\alpha}[\tau_n] = W_{\alpha}(E).
\end{align}

We now follow a standard potential theoretic argument to show that $\tau=\mu_E.$ Let $K_M(x,y) := \min\left(k_{\alpha}(x-y),M\right).$ It is clear that $K_M(x,y)$ is a continuous function in $x$ and $y$, and that $K_M(x,y)$ increases to
$k_{\alpha}(x-y)$ as $M\to\infty.$ Using the Monotone Convergence Theorem and the weak* convergence of $\hat\tau_n\times\hat\tau_n$ to $\tau\times\tau,$ we obtain for the energy of $\tau$ that
\begin{align*}
I_{\alpha}[\tau] &= \iint k_{\alpha}(x-y)\,d\tau(x)\,d\tau(y) =
\lim_{M\to\infty} \left( \lim_{n\to\infty} \iint K_M(x,y)\,
d\hat\tau_n(x)\,d\hat\tau_n(y) \right) \\ &\le \lim_{M\to\infty} \left(
\lim_{n\to\infty} \left( \frac{2}{n^2} \sum_{1\le j<k\le m_n} K_M(x_{j,n},x_{k,n}) + \frac{M}{n} \right) \right) \\ &\le
\lim_{M\to\infty} \left( \liminf_{n\to\infty} \frac{2}{n^2}
\sum_{1\le j<k\le m_n} k_{\alpha}(x_{j,n}-x_{k,n}) \right) \\ &= \liminf_{n\to\infty} \frac{m_n(m_n-1)}{n^2}
\hat{I}_{\alpha}[\hat \tau_n] \le W_{\alpha}(E),
\end{align*}
where we applied \eqref{5.5} and $\lim_{n\to\infty} m_n/n = 1$ in the last estimate. Recall that $\supp(\tau) \subset E$ for $0<\alpha<2$ by \eqref{5.3}. Since $I_{\alpha}[\nu]>W_{\alpha}(E)$ for any probability measure $\nu\neq\mu_{E},\ \supp(\nu)\subset E$, we obtain that $\tau=\mu_E$ and \eqref{1.3} follows for $0<\alpha<2$. In the case $\alpha=2,$ we have that $\supp(\tau) \subset \hat E = \R^d\setminus\Omega_E,$ where $W_{\alpha}(\hat E)=W_{\alpha}(E)$ and $\mu_{\hat E}=\mu_E$ by \cite[p. 164]{La}. Since again $I_{\alpha}[\nu]> W_{\alpha}(\hat E)$ for any probability measure $\nu\neq\mu_{\hat E},\ \supp(\nu)\subset \hat E$, we conclude that $\tau=\mu_{\hat E}=\mu_E$ as before.

Let us turn to the converse statement \eqref{1.3} $\Rightarrow$ \eqref{1.2}. Note that $g_E(x)\le W_{\alpha}(E)$ for all $x\in\R^d,$ cf. \eqref{5.2}. Choosing $R>0$ so large that $E\subset B_R$, we obtain from \eqref{1.3} that
\[
\frac{1}{n} \sum_{|x_{k,n}|\ge R} g_E(x_{k,n}) \le \frac{o(n)}{n}\,W_{\alpha}(E),
\]
which implies that
\begin{align} \label{5.6}
\limsup_{n\to\infty} \frac{1}{n} \sum_{|x_{k,n}|\ge R} g_E(x_{k,n}) \le 0.
\end{align}
Since $g_E(x)$ is upper semi-continuous  in $\R^d,$ we obtain from \eqref{1.3} and Lemma 0.1 of \cite[p. 8]{La} that
\begin{align} \label{5.7}
\limsup_{n\to\infty} \frac{1}{n} \sum_{|x_{k,n}| < R} g_E(x_{k,n}) &= \limsup_{n\to\infty} \int_{B_R} g_E(x)\,d\tau_n(x) \le \int_{B_R} g_E(x)\,d\mu_E(x) \\ \nonumber &= W_{\alpha}(E) - \int U_{\alpha}^{\mu_E}(x)\,d\mu_E(x) = W_{\alpha}(E) - I_{\alpha}[\mu_E] = 0.
\end{align}
Observe from the definition of $m_E(X_n)$ and \eqref{5.6}-\eqref{5.7} that
\begin{align*}
0 \le \liminf_{n\to\infty} m_E(X_n)\le \limsup_{n\to\infty} m_E(X_n) \le \limsup_{n\to\infty} \frac{1}{n} \sum_{k=1}^n g_E(x_{k,n}) \le 0,
\end{align*}
so that \eqref{1.2} follows.
\end{proof}

\begin{proof}[Proof of Theorem \ref{thm2.1}]
Given $r>0$, define the measures $\nu_k^r$ with $d\nu_k^r(x_{k,n} + ry) = dS(y)/\omega_d,\ y\in S,$ where $dS$ denotes the surface area measure on the unit hypersphere $S$ in $\R^d$. Let $\tau_n:=\tau(X_n)$ and
\[
\tau_n^r:=\frac{1}{n}\sum_{k=1}^n \nu_k^r,
\]
and estimate
\begin{align} \label{5.9}
\left|\int\phi\,d\tau_n - \int\phi\,d\tau_n^r\right| \le \frac{1}{n}\sum_{k=1}^n \frac{1}{\omega_d}\int_S \left|\phi(x_{k,n}) - \phi(x_{k,n} + ry)\right|\,dS(y) \le \omega(\phi;r).
\end{align}
We now assume that $E$ is a regular set bounded by finitely many piecewise smooth $(d-1)$-dimensional surfaces, and remove this assumption in the end of proof. Since $E$ is regular, we have that  $g_E(x) = 0,\ x\in\R^d\setminus\Omega_E.$ Consider the signed measure $\lambda:=\tau_n^r-\mu_E,\ \lambda(\R^d)=0.$ This measure is recovered from its potential by the formula
\[
d\lambda=-\frac{1}{(d-2)\omega_d}\left(\frac{\partial U^{\lambda}}{\partial n_+} + \frac{\partial U^{\lambda}}{\partial n_-}\right) dS,
\]
where $dS$ is the surface area on $\supp(\lambda) = \supp(\mu_E) \cup \left( \cup_{k=1}^n \{x:|x-x_{k,n}|=r\}\right)$, and $n_{\pm}$ are the inner and the outer normals, see \cite[p. 164]{Ke} and \cite[pp. 164--165]{La}. Let $B_R:=\{x:|x|<R\}$ be a ball containing the support of $\phi.$ We use Green's identity
\[
\int_G u \Delta v\,dV =  \int_{\partial G} u\,\frac{\partial v}{\partial n}\,dS - \int_G \nabla u \cdot \nabla v\,dV
\]
with $u=\phi$ and $v=U^{\lambda}$ in each connected component $G$ of $B_R\setminus\supp(\lambda).$ Since $U^{\lambda}$ is harmonic in $G$, we have that $\Delta U^{\lambda}=0$ in $G$. Adding Green's identities for all domains $G$, we obtain that
\begin{align} \label{5.10}
\left|\int\phi\,d\lambda\right| = \frac{1}{(d-2)\omega_d} \left| \int_{B_R} \nabla \phi \cdot \nabla U^{\lambda} \,dV \right| \le \frac{1}{(d-2)\omega_d} \sqrt{D[\phi]}\,\sqrt{D[U^{\lambda}]},
\end{align}
by the Cauchy-Schwarz inequality. It is known that $D[U^{\lambda}] = (d-2)\omega_d  I[\lambda]$ \cite[Thm 1.20]{La}, where $I[\lambda]=\iint |x-y|^{2-d}\,d\lambda(x)\,d\lambda(y) = \int U^{\lambda}\,d\lambda$ is the energy of $\lambda$. We also recall that $\int U^{\mu_E}\,d\mu_E = I[\mu_E] = W(E)$, which gives that
\[
I[\lambda]=\int U^{\tau_n^r}\,d\tau_n^r - 2\int U^{\mu_E}\,d\tau_n^r + W(E).
\]
Since $g_E(x)$ is harmonic in $\Omega_E$, the mean value property implies that
\begin{align*}
-\int U^{\mu_E}\,d\tau_n^r &= \int \left( g_E(x) - W(E) \right)\, d\tau_n^r(x) \\ &= \frac{1}{n} \left(\sum_{d_E(x_{k,n})\le r} \int g_E\,d\nu_k^r + \sum_{d_E(x_{k,n})>r} \int g_E\,d\nu_k^r \right) - W(E) \\ &\le \frac{1}{n} \left(\sum_{d_E(x_{k,n})\le r} \max_{d_E(x)\le 2r} g_E(x) + \sum_{d_E(x_{k,n})>r} g_E(x_{k,n})\right) - W(E) \\ &\le\max_{d_E(x)\le 2r} g_E(x) + m_E(X_n) - W(E).
\end{align*}
Taking into account the representation \cite[p. 165]{La}
\[
U^{\nu_k^r}(x)=(\max(r,|x-x_{k,n}|))^{2-d},\quad x\in\R^d,
\]
we proceed further with
\begin{align*}
\int U^{\tau_n^r}\,d\tau_n^r &= \frac{1}{n^2} \sum_{j,k=1}^n \int U^{\nu_k^r}\,d\nu_j^r \le \frac{1}{n^2} \left(\sum_{j\neq k} |x_{j,n}-x_{k,n}|^{2-d} + n r^{2-d}\right) \\ &= \frac{n-1}{n} \hat{I}[\tau_n] + \frac{r^{2-d}}{n},
\end{align*}
and combine the energy estimates to obtain
\[
I[\lambda] \le 2m_E(X_n) + \frac{n-1}{n} \hat{I}[\tau_n] - W(E) + \frac{r^{2-d}}{n} + 2\max_{d_E(x)\le 2r} g_E(x).
\]
Using \eqref{5.9}, \eqref{5.10} and the above estimate, we deduce \eqref{2.1}-\eqref{2.2} by the following argument:
\begin{align*}
\left|\int\phi\,d\tau_n - \int\phi\,d\mu_E\right| &\le \left|\int\phi\,d\tau_n - \int\phi\,d\tau_n^r\right| + \left|\int\phi\,d\tau_n^r - \int\phi\,d\mu_E\right| \\ &\le \omega(\phi;r) + \frac{\sqrt{D[\phi]}\sqrt{D[U^{\lambda}]}}{(d-2)\omega_d} = \omega(\phi;r) + \sqrt{\frac{D[\phi]}{(d-2)\omega_d}}\,\sqrt{I[\lambda]}.
\end{align*}
Thus we proved the result for regular sets bounded by finitely many piecewise smooth surfaces. To show that \eqref{2.1}-\eqref{2.2} hold for an arbitrary compact set $E$  of positive capacity, we approximate $E$ by a decreasing sequence $E_m,\ m\in\N,$  of compact sets with piecewise smooth boundaries. Let $\eps_1=1$ and consider an open cover of $E$ by the balls $\{B(x,\eps_1)\}_{x\in E},$ where $B(x,\eps_1)$ is centered at $x$ and has radius $\eps_1.$ There exists a finite subcover such that
$E\subset\cup_{k=1}^{N_1} B(c_{k,1},\eps_1).$ Define $E_1:=\cup_{k=1}^{N_1} \overline B(c_{k,1},\eps_1).$ We construct the sets $E_m$ inductively for $m\ge 2.$ Set $\eps_m:={\rm dist}(E,\partial E_{m-1})/2 >0.$ As before, we have a finite subcover such that
\[
E\subset\bigcup_{k=1}^{N_m} B(c_{k,m},\eps_m),\quad m\in\N,
\]
where $c_{k,m}\in E,\ k=1,\ldots,N_m.$ Let
\[
E_m := \bigcup_{k=1}^{N_m} \overline{B(c_{k,m},\eps_m)},\quad m\in\N,
\]
and note that $E_m\subset E_{m-1}$ and $\eps_m\le\eps_{m-1}/2,\ m\ge 2.$ Clearly, the boundary of every $E_m$ consists of finitely many piecewise smooth surfaces, and each surface is composed of finitely many spherical fragments. Thus every $E_m$ is regular by Theorem 6.6.15 of \cite[p. 185]{AG}, and \eqref{2.1}-\eqref{2.2} hold for every $E_m,\ m\in\N.$ Observe that $\lim_{m\to\infty} \eps_m = 0,$ so that
\[
E=\bigcap_{m=1}^{\infty} E_m.
\]
If $g_{E_m}(x)$ is the Green function for $\R^d\setminus E_m$, then
\[
g_{E_m}(x)\le g_{E_{m+1}}(x)\le g_E(x),\quad x\in\R^d,
\]
for any $m\in\N,$ by the Maximum Principle. This gives that
\[
\max_{d_{E_m}(x)\le 2r} g_{E_m}(x) \le \max_{d_{E_m}(x)\le 2r} g_E(x),\quad m\in\N.
\]
Since $g_E(x)$ is subharmonic in $\R^d$ and harmonic in $\Omega_E$, the maximum on the right of the above inequality is attained on the set $\{x\in\R^d:d_{E_m}(x)=2r\}\subset\Omega_E.$  We have that
\[
\lim_{m\to\infty} \max_{d_{E_m}(x)=2r} g_E(x) = \max_{d_E(x)=2r} g_E(x),
\]
because $d_{E_m}(x) \le d_E(x) \le d_{E_m}(x) + \eps_m,\ x\in\R^d,$ by the triangle inequality. Thus
\begin{align} \label{5.11}
\limsup_{m\to\infty} \max_{d_{E_m}(x)\le 2r} g_{E_m}(x) \le \max_{d_E(x)\le 2r} g_E(x).
\end{align}
Furthermore, Harnack's Theorem implies that
\[
\lim_{m\to\infty} g_{E_m}(x) = g_E(x),\quad x\in\Omega_E,
\]
so that
\begin{align} \label{5.12}
\lim_{m\to\infty} m_{E_m}(X_n) = m_E(X_n).
\end{align}
Using Helley's selection theorem, we assume that $\mu_{E_m} \stackrel{*}{\rightarrow} \mu$ as $m\to\infty$ along a subsequence $N\subset\N.$ Then we have that
\[
I[\mu] \le \liminf_{m\in N} I[\mu_{E_m}]
\]
by \cite[p. 78]{La}. On the other hand, it is known \cite[pp. 140--141]{La} that
\begin{align} \label{5.13}
\lim_{m\to\infty} W(E_m) = W(E),
\end{align}
which gives that
\[
I[\mu] \le \liminf_{m\in N} W(E_m) = W(E).
\]
Note that $\mu$ is a unit measure supported on $\partial\Omega_E\subset E$ by our construction. Since $I[\mu] \le W(E)$, we conclude that $\mu=\mu_E$ by uniqueness of the equilibrium measure minimizing the energy functional.
This argument holds for any subsequence $N$, which means that
$\mu_{E_m} \stackrel{*}{\rightarrow} \mu_E$ as $m\to\infty.$ Consequently,
\[
\lim_{m\to\infty} \int \phi\,d\mu_{E_m} = \int \phi\,d\mu_E.
\]
We now pass to the limit in \eqref{2.1} stated for $E_m$, as $m\to\infty$, and use the above equation together with \eqref{5.11}, \eqref{5.12} and \eqref{5.13} to prove that \eqref{2.1}-\eqref{2.2} also hold for $E$.

\end{proof}

\begin{proof}[Proof of Theorem \ref{thm2.2}] One readily finds from the triangle inequality that
\[
\left| |x-t_1| - |x-t_2| \right| \le |t_1-t_2|, \quad x,t_1,t_2\in\R^d,
\]
and
\[
\left| d_E(t_1) - d_E(t_2) \right| \le |t_1-t_2|, \quad t_1,t_2\in\R^d.
\]
Given a fixed point $y\in E^c$, we have
\begin{align} \label{5.14}
d_E(y) \le |y-x| + d_E(x),\quad x\in\R^d,\ y\in E^c.
\end{align}
Let $\diam(E):=\max_{t,w\in E} |t-w|$ be the diameter of $E$, and set $R:=\diam(E)+d_E(y)+1.$ We apply Theorem \ref{thm2.1} with the function
\begin{align} \label{5.15}
\phi(x):=\max\left((|y-x|+d_E(x))^{2-d}-R^{2-d},0\right),\quad x\in\R^d,\ y\in E^c.
\end{align}
It is clear that $\supp(\phi)\subset B(y,R):=\{x\in\R^d:|x-y|<R\}$.
Furthermore, $E\subset\supp(\phi)$ because
\[
|y-x| \le d_E(y) + \diam(E) < R,\quad x\in E,\ y\in E^c,
\]
by the triangle inequality. Since $d_E$ is Lipschitz continuous, the function $f(x):=|y-x|+d_E(x),\ x\in\R^d,$ satisfies the Lipschitz condition
\[
\left| f(t_1) - f(t_2) \right| \le 2 |t_1-t_2|, \quad t_1,t_2\in\R^d.
\]
Thus all first order partial derivatives of $f$ exist a.e. with respect to the volume measure, and we obtain that
\[
\left|\frac{\partial f}{\partial x_i}(x)\right| \le 2, \quad i=1,\ldots,d,\quad \mbox{for a.e. } x=(x_1,\ldots,x_d)\in\R^d.
\]
It follows that $\phi$ is Lipschitz continuous and that $\partial\phi/\partial x_i$ also exist a.e. in the same sense as above, with
\[
\left|\frac{\partial \phi}{\partial x_i}(x)\right| \le \frac{2(d-2)}{(|y-x|+d_E(x))^{d-1}} \le \frac{2(d-2)}{(d_E(y))^{d-1}}, \quad i=1,\ldots,d,
\]
for a.e. $x=(x_1,\ldots,x_d)\in\R^d$ by \eqref{5.14}. This gives the estimates
\[
|\phi(t_1)-\phi(t_2)|\le |t_1-t_2|\, \sup_{x\in\R^d} |\nabla\phi(x)| \le \frac{2(d-2)\sqrt{d}}{(d_E(y))^{d-1}}\, |t_1-t_2|,\quad t_1,t_2\in\R^d,
\]
and
\begin{align} \label{5.16}
\omega(\phi;r) \le \frac{2(d-2)\sqrt{d}}{(d_E(y))^{d-1}}\, r.
\end{align}
Furthermore, we obtain for the Dirichlet integral
\begin{align*}
D[\phi] &= \int_{\R^d} |\nabla\phi|^2 \,dV \le \int_{B(y,R)} \frac{4d(d-2)^2\,dV(x)}{(|y-x|+d_E(x))^{2(d-1)}} \\ &\le \int_{B(y,d_E(y))} \frac{4d(d-2)^2\,dV(x)}{(|y-x|+d_E(x))^{2(d-1)}} + \int_{d_E(y)\le |y-x|\le R} \frac{4d(d-2)^2\,dV(x)}{(|y-x|+d_E(x))^{2(d-1)}} \\ &\le O\left( d_E(y)^{2-d}\right) + O\left(\int_{d_E(y)}^R \frac{r^{d-1}\,dr}{r^{2(d-1)}}\right) = O\left(d_E(y)^{2-d}\right)
\end{align*}
by $\supp(\phi)\subset B(y,R)$ and \eqref{5.14}. We now let $r=n^{-p/s},$ and obtain that $\omega(\phi;r)=O((d_E(y))^{1-d} n^{-p/s})$ by \eqref{5.16}. Since the Green function $g_E$ satisfies the H\"older condition \eqref{2.3},  we have that
\[
\max_{d_E(x)\le 2r} g_E(x) \le O(n^{-p}).
\]
Applying the above estimates and \eqref{2.4} in \eqref{2.1}-\eqref{2.2}, we arrive at
\begin{align} \label{5.17}
\left|\int\phi\,d\mu_E - \frac{1}{n} \sum_{k=1}^n \phi(x_{k,n})\right| &\le O\left(d_E(y)^{1-d} n^{-p/s}\right) + O\left(d_E(y)^{2-d}\right)^{1/2} \left( O\left(n^{-p}\right) \right)^{1/2} \\ \nonumber &\le O\left(d_E(y)^{1-d} n^{-p/s} + d_E(y)^{1-d/2} n^{-p/2} \right)\quad \mbox{as }n\to\infty,
\end{align}
where we also used that $m_E(X_n)=0$. Note that all constants in $O$ terms are independent of the point $y\in E^c$, of the set $X_n$, as well as of $n\ge 2.$ It remains to observe that $\phi(x)=|y-x|^{2-d} - R^{2-d}$ for $x\in E$, so that
\begin{align*}
\int\phi\,d\mu_E - \frac{1}{n} \sum_{k=1}^n \phi(x_{k,n}) &= U^{\mu_E}(y) - U^{\tau(X_n)}(y),\quad y\in E^c.
\end{align*}
Thus \eqref{2.5} follows from \eqref{5.17}.

Let $q>0$. If $y\in\Gamma:=\{x\in\Omega_E: g_E(x)=n^{-q}\}$ then $d_E(y) \ge (n^{-q}/A(E))^{1/s}$ by \eqref{2.3}, and we obtain from \eqref{2.5} that
\begin{align} \label{5.18}
U^{\mu_E}(y) - U^{\tau(X_n)}(y) \le O\left(n^{q(d-1)/s-p/s}\right) + O\left(n^{q(d-2)/(2s)-p/2}\right),\quad y\in\Gamma.
\end{align}
Recall that $U^{\mu_E}(y)=W(E)-g_E(y)=W(E)-n^{-q}$ for $y\in\Gamma.$ Hence
\[
U^{\tau(X_n)}(y) \ge W(E)-n^{-q}-O\left(n^{q(d-1)/s-p/s}\right)- O\left(n^{q(d-2)/(2s)-p/2}\right),\quad y\in\Gamma.
\]
We can now choose $q=q(d,s)>0$ so small that
\[
U^{\tau(X_n)}(y) \ge W(E) - O\left(n^{-q}\right),\quad y\in\Gamma.
\]
Observe that the open set $G:=\{x\in\R^d: g_E(x)<n^{-q}\}$ contains $E$ strictly inside. Since $U^{\tau(X_n)}$ is superharmonic in $\R^d$, it attains minimum over $G$ on its boundary $\Gamma$. It follows that
\[
U^{\tau(X_n)}(x) \ge \inf_{y\in\Gamma} U^{\tau(X_n)}(y) \ge W(E) - O\left(n^{-q}\right),\quad x\in E.
\]
Since $U^{\mu_E}(x)=W(E),\ x\in E,$ we obtain that
\[
U^{\mu_E}(x) - U^{\tau(X_n)}(x) = O\left(n^{-q}\right),\quad x\in E.
\]
The Principle of Domination \cite[p. 110]{La} implies that the above equation holds for all $x\in\R^d$, because $\supp(\mu_E)\subset E.$
\end{proof}

\begin{proof}[Proof of Corollary \ref{cor2.3}]
We first observe that the Fekete points $\cF_n$ satisfy
\begin{align} \label{5.19}
\hat I[\tau(\cF_n)] \le W(E), \quad n\ge 2,
\end{align}
This fact holds because the discrete energies of Fekete sets increase to $W(E)$ with $n,$ see \cite[p. 160]{La}. Hence \eqref{2.4} holds true and Theorem \ref{2.2} applies to $\cF_n$.

It turns out that \eqref{5.19} is also true for the Leja points $\mathcal L_n=\{\xi_k\}_{k=0}^{n-1},\ n\in\N.$ Consider the corresponding potentials $U^{\tau(\mathcal L_n)}$, and recall that
\[
\min_{x\in E} U^{\tau(\mathcal L_n)}(x) = U^{\tau(\mathcal L_n)}(\xi_n) = \frac{1}{n} \sum_{k=0}^{n-1} |\xi_n-\xi_k|^{2-d}
\]
by definition. Hence we have for the discrete energy
\begin{align*}
\hat I[\tau(\mathcal L_n)] &= \frac{2}{n(n-1)} \sum_{0\le j<k\le n-1} |\xi_j-\xi_k|^{2-d} = \frac{2}{n(n-1)} \sum_{k=1}^{n-1} k\,U^{\tau(\mathcal L_k)}(\xi_k) \\ &= \frac{2}{n(n-1)} \sum_{k=1}^{n-1} k\,\min_{x\in E} U^{\tau(\mathcal L_k)}(x).
\end{align*}
Since the inequality
\[
\min_{x\in E} U^{\nu}(x) \le W(E)
\]
holds for the potential of any positive unit measure $\nu$, see Theorem 2.3 of \cite[p. 138]{La}, we obtain that
\[
\hat{I}[\tau(\mathcal L_n)] \le W(E).
\]
\end{proof}

\end{document}